\documentclass[10pt,a4paper]{article}

\usepackage{amsmath,amssymb,amsfonts,amsthm}
\usepackage{mathrsfs} 
\usepackage{bm} 
\usepackage{graphicx} 
\usepackage{xcolor}
\usepackage{booktabs}
\usepackage{caption}
\usepackage{hyperref}
\usepackage{geometry} 
\usepackage{algorithm}
\usepackage{algorithmicx}
\usepackage{algpseudocode}

\usepackage{listings}

\theoremstyle{plain}
\newtheorem{theorem}{Theorem}[section] 

\newtheorem{lemma}[theorem]{Lemma}

\theoremstyle{definition}
\newtheorem{definition}[theorem]{Definition}

\newtheorem{assumption}[theorem]{Assumption}
\theoremstyle{remark}

\title{Stochastic Maximum Principles and Linear-Quadratic Optimal Control Problems for Fractional Backward Stochastic Evolution Equations in Hilbert Spaces}
\author{Javad A. Asadzade\thanks{Department of Mathematics, Eastern Mediterranean University, North Cyprus, Turkey. Email: javad.asadzade@emu.edu.tr} 
	\, and \, Nazim I. Mahmudov\thanks{Department of Mathematics, Eastern Mediterranean University, North Cyprus, Turkey; Research Center of Econophysics, Azerbaijan State University of Economics (UNEC), Baku, Azerbaijan. Email: nazim.mahmudov@emu.edu.tr}}
\date{} 

\begin{document}
	
	\maketitle
	
	\begin{abstract}
		This paper develops a comprehensive framework for optimal control of systems governed by fractional backward stochastic evolution equations (FBSEEs) in Hilbert spaces. We first establish a stochastic maximum principle (SMP) as a necessary condition for optimality. This is achieved by introducing spike variations, deriving precise estimates for the associated variational equations, and constructing an adjoint process tailored to the fractional dynamics. Subsequently, we apply this general principle to solve the linear-quadratic (LQ) optimal control problem explicitly. The resulting optimal control is characterized in closed form via the adjoint process and is shown to be governed by a system of coupled fractional forward-backward stochastic equations. Our work bridges fractional calculus with stochastic control theory, providing a rigorous foundation for controlling infinite-dimensional systems with memory and long-range dependencies.
	\end{abstract}
	
	
	\noindent\textbf{Mathematics Subject Classification (MSC 2020):} 93E20, 49K20, 49J15.
	\section{Introduction}\label{sec1}
	
	The theory of stochastic optimal control has long been a cornerstone of modern applied mathematics, with profound applications in finance, engineering, physics, and biological sciences. Its historical roots trace back to the pioneering contributions of Kushner \cite{Kushner, Kushner1964} and Bismut \cite{Bismut}, who established the foundations of the SMP. The SMP provides necessary and, under suitable conditions, sufficient conditions for optimality in systems governed by stochastic differential equations (SDEs). A major milestone was the introduction of backward stochastic differential equations (BSDEs) by Pardoux and Peng \cite{Pardoux}, which were subsequently developed by Peng \cite{Peng, Peng2, Peng1990} and became a fundamental analytical tool in stochastic control theory.
	
	Building on these foundations, the 1980s and 1990s saw substantial progress in extending the stochastic maximum principle to more general settings. In particular, Bensoussan \cite{Bensoussan} formulated the SMP for distributed parameter systems, opening the way for the analysis of stochastic partial differential equations (SPDEs). Later, Mahmudov \cite{Mahmudov2002} derived a SMP for stochastic evolution systems in Hilbert spaces, generalizing the theory to infinite-dimensional settings. These developments, together with the earlier contributions of Bismut \cite{Bismut}, laid the groundwork for the study of optimal control in infinite-dimensional stochastic environments.
	
	Further progress was achieved by Hu and Peng \cite{HuPeng1991}, who extended BSDEs to semilinear stochastic evolution equations, and by Mahmudov and McKibben \cite{MahmudovMcKibben2007}, who established a unified framework for BSEEs in Hilbert spaces. Their work derived first-order stochastic evolution systems and introduced optimal control formulations for such equations, significantly enriching the mathematical structure of stochastic control in function spaces. Lü and Zhang \cite{Zhang2014} later advanced the theory by formulating a general Pontryagin-type stochastic maximum principle for BSEEs in infinite dimensions.
	
	Parallel research addressed the relaxation of smoothness assumptions and the inclusion of more general stochastic dynamics. Mao \cite{Mao1995}, Rong \cite{Rong1997}, and Situ \cite{Situ2002} studied BSDEs with jumps and non-Lipschitz coefficients, while Chen and Wang \cite{ChenWang2000} explored infinite-horizon BSDEs. Briand and Hu \cite{BriandHu2006, BriandHu2008} treated BSDEs with quadratic growth and unbounded terminal conditions, broadening existence and uniqueness results beyond standard Lipschitz frameworks. Collectively, these contributions significantly expanded the applicability of BSDE theory.
	
	In the field of stochastic control, El Karoui, Peng, and Quenez \cite{El} integrated BSDEs into financial mathematics, while Morlais \cite{Morlais} extended them to utility optimization via quadratic BSDEs. Cadenillas and Haussmann \cite{Cadenillas}, and later Cadenillas and Karatzas \cite{Cadenillas1995}, addressed stochastic maximum principles in singular control problems, whereas Li \cite{Li2012}, Buckdahn et al. \cite{Buckdahn, Buckdahn2016}, and Yong \cite{Yong2013} developed SMPs for mean-field and interactive systems. Together, these works form the theoretical backbone of modern stochastic control.
	
	In recent years, increasing attention has been devoted to fractional stochastic systems, which capture memory and hereditary effects absent from classical models. Fractional differential equations based on Caputo or Riemann–Liouville operators provide an appropriate mathematical framework for such systems \cite{ZhouShenZhang2013, Zhou2014}. Within the stochastic setting, Li and Wang \cite{LiWang2019} and Yang and Gu \cite{YangGu2021} established existence, uniqueness, and asymptotic properties for fractional stochastic evolution equations. Mahmudov and Ahmadova \cite{Nazim}, along with Ahmadova and Mahmudov \cite{Arzu}, extended the analysis to fractional backward stochastic differential equations and their singular variants, revealing new challenges associated with nonlocal and weakly singular dynamics.
	
	The development of FBSEEs represents a synthesis of advances in fractional calculus and backward stochastic analysis. Li and Luo \cite{Li} provided new well-posedness results for FBSEEs, while Asadzade and Mahmudov \cite{Cavad1} studied singular mean-field backward and forward stochastic Volterra integral equations in infinite-dimensional spaces, establishing existence, uniqueness, and adapted M-solutions. Their results also yield fractional variants of the Pontryagin maximum principle, paving the way for stochastic optimal control in fractional and memory-dependent settings.
	
	Within stochastic control, the LQ optimal control problem continues to be one of the most analytically tractable and practically relevant frameworks. The classical stochastic LQ problem has been extensively investigated by Hu, Jin, and Zhou \cite{Hu}, Yong \cite{Yong2013}, and Meng and Shi \cite{MengShi2013}. Extensions to mean-field and fractional systems \cite{Li2012, Buckdahn, Buckdahn2016} have further enriched the field. Nonetheless, the LQ theory for FBSEEs remains underdeveloped, primarily due to the intricate coupling between fractional derivatives, infinite-dimensional operators, and backward stochastic dynamics.
	
	This paper is principally devoted to the establishment of a SMP and the subsequent analysis of LQ optimal control problems for a class of FBSEEs. In undertaking this endeavor, we systematically extend the seminal results of Bismut \cite{Bismut} and Peng \cite{Peng, Peng1990} into a unified framework that simultaneously incorporates both fractional-order dynamics and infinite-dimensional state spaces, thereby yielding novel theoretical insights and a set of necessary optimality conditions. Our methodological approach is conceptually rooted in the first-order stochastic evolution framework pioneered by Mahmudov and McKibben \cite{MahmudovMcKibben2007}, whose analytical foundation we generalize to the fractional setting.
	
	Building upon the recent well-posedness results for FBSEEs established by Li and Luo \cite{Li}, we formulate and rigorously analyze their fractional counterparts to derive a new SMP and explicit LQ control representations. By synthesizing techniques from the theory of fractional BSDEs with the analytical machinery developed for singular Volterra-type equations \cite{Arzu, Cavad1, Nazim}, we construct a cohesive theoretical framework. This synthesis effectively bridges the disparate domains of fractional stochastic analysis and modern control theory, consequently extending the purview of classical stochastic control to encompass systems characterized by non-Markovian memory effects and infinite-dimensional dynamics.
	
	The remainder of this paper is organized as follows. Section 2 presents the mathematical preliminaries, including fractional calculus operators, Wright functions, and the construction of fractional resolvent operators for stochastic evolution equations. Section 3 establishes a SMP for general fractional backward stochastic systems: we introduce spike variations of the optimal control, derive precise moment estimates for the variational equations, prove a first-order variation formula for the cost functional, and obtain necessary optimality conditions via an adjoint process. Section 4 specializes to the linear-quadratic case, where we solve the associated backward LQ problem explicitly by deriving the adjoint equation through fractional integration by parts and obtaining a closed-form representation of the optimal control. Finally, Section 5 provides concluding remarks.
	
	\section{Preliminaries}
	This section establishes the fundamental mathematical framework and analytical notation required for our subsequent developments. We begin by considering a complete probability space $(\Omega, \mathcal{F}, \mathbb{P})$ endowed with a normal filtration $\{\mathcal{F}_t\}_{0 \le t \le b}$, and introduce separable Hilbert spaces $H$, $U$, and $E$. The stochastic analysis is built around a $Q$-Wiener process $B = \{B_t, t \in [0,b]\}$ defined on this probability space, characterized by a linear, bounded, nonnegative covariance operator $Q \in \mathcal{L}(E)$ with finite trace ($\operatorname{Tr} Q < \infty$).
	
	The spectral structure of the noise process is specified through a complete orthonormal system $\{e_k\}_{k \ge 1}$ in $E$ and a bounded sequence of nonnegative eigenvalues $\{\lambda_k\}_{k \ge 1}$ satisfying $Q e_k = \lambda_k e_k$. This facilitates the representation of the Wiener process via independent standard Brownian motions $\{\beta_k\}_{k \ge 1}$:
	\[
	\langle B_t, e \rangle_E = \sum_{k=1}^{\infty} \sqrt{\lambda_k}\, \langle e_k, e \rangle_E\, \beta_k(t), \quad e \in E, \; t \in [0,b],
	\]
	where $\langle \cdot , \cdot \rangle_E$ denotes the inner product in $E$. The filtration $\{\mathcal{F}_t\}$ is assumed to be generated by this Wiener process.
	
	Key function spaces are defined as follows: $L_2^0 = L_2(Q^{1/2}E,H)$ represents the space of Hilbert-Schmidt operators from $Q^{1/2}E$ to $H$, endowed with the inner product $\langle \Psi, \Phi \rangle_{L_2^0} = \operatorname{Tr}(\Psi Q \Phi^*)$. We further define $L^2(\Omega, \mathcal{F}_b; H)$ as the Hilbert space of $\mathcal{F}_b$-measurable, square-integrable $H$-valued random variables, and $L^2_{\mathcal{F}}([0,b]; H)$ as the space of square-integrable, $\{\mathcal{F}_t\}$-adapted $H$-valued processes with norm
	\[
	\|x\|_{L^2_{\mathcal{F}}([0,b];H)} = \left(\mathbb{E} \int_0^b \|x_t\|_H^2\, dt\right)^{\!1/2}<\infty.
	\]
	Finally, $\mathcal{L}(K,H)$ denotes the space of bounded linear operators from $K$ to $H$ with the standard operator norm, with the special case $\mathcal{L}(H) := \mathcal{L}(H,H)$.
	\begin{definition}\cite{Zhou2014}
		The \emph{Gamma function} is defined for $x>0$ by
		\[
		\Gamma(x) = \int_0^\infty t^{x-1} e^{-t} \, dt.
		\]
	\end{definition}
	\begin{definition}\cite{Zhou2014}
		The left Caputo fractional derivative of order $\alpha \in (0,1)$ with lower limit $0$ for a function $f$ is defined by
		\[
		{^{C}_0}D^{\alpha}(t) \kappa_t = \frac{1}{\Gamma(1-\alpha)} \int_0^t (t-s)^{-\alpha} \kappa^{\prime}_s \, ds, \quad t>0.
		\]
	\end{definition}
	\begin{definition}\cite{Zhou2014}
		The right Caputo fractional derivative of order $\alpha \in (0,1)$ with upper limit $b$ for a function $f$ is defined by
		\[
		{^{C}_t}D^{\alpha}_b \kappa_t = -\frac{1}{\Gamma(1-\alpha)} \int_t^b (s-t)^{-\alpha} \kappa^{\prime}_s \, ds, \quad t>0.
		\]
	\end{definition}
	\begin{definition}\cite{Zhou2014}
		The \emph{Wright function} $W_{\rho}$ is defined as
		\[
		W_{\rho}(r)=\sum_{k=1}^{\infty} \frac{(-r)^{k-1}}{(k-1)!\Gamma(1-\rho k)},\quad r\in \mathbb{C},\, \rho\in (0,1),
		\]
		and it satisfies
		\[
		\int_0^{\infty}r^{\gamma}W_{\rho}(r)\, dr=\frac{\Gamma(1+\gamma)}{\Gamma(1+\rho \gamma)},\quad \gamma>-1.
		\]
	\end{definition}
	
	\begin{definition}\cite{Zhou2014}\label{def:resolvent}
		Let $A$ be the infinitesimal generator of a $C_0$-semigroup $\{S_t\}_{t \ge 0}$ on a Hilbert space $H$.  
		For $\alpha \in (0,1)$, the operators $S_{\alpha}(t)$ and $P_{\alpha}(t)$, called the fractional resolvent operators, are defined by
		\[
		S_{\alpha}(t) = \int_0^{\infty} \alpha\, r\, W_{-\alpha}(r)\, S(t^{\alpha} r)\, dr, \quad
		P_{\alpha}(t) = \int_0^{\infty} W_{-\alpha}(r)\, S(t^{\alpha} r)\, dr, \quad t > 0,
		\]
		where $W_{-\alpha}(r)$ is the Wright function given by
		\[
		W_{-\alpha}(r) = \sum_{k=1}^{\infty} \frac{(-r)^{k-1}}{(k-1)! \, \Gamma(1+\alpha k)}, \quad r \in \mathbb{C}, \, \alpha \in (0,1).
		\]
	\end{definition}
	
	\begin{lemma}[\cite{Li}]\label{lem2}
		For any fixed $t \in [0,b]$, $S_{\alpha}(t)$ and $P_{\alpha}(t)$ are linear operators satisfying
		\[
		\| S_{\alpha}(t) \|_{\mathcal{L}(H)} \leq M_1, \qquad 
		\frac{\alpha M_0}{\Gamma(1+\alpha)} \leq \| P_{\alpha}(t) \|_{\mathcal{L}(H)} \leq \frac{\alpha M_1}{\Gamma(1+\alpha)}.
		\]
	\end{lemma}
	
	\section{Stochastic maximum principle}
	In this section we consider the following stochastic controlled system:
	\begin{align}\label{e1}
		x_t=S_{\alpha}(b-t)x_b&+\int_t^b (s-t)^{\alpha-1}P_{\alpha}(s-t)\kappa(s,x_s,z_s,u_s)\, ds\nonumber\\
		&+\int_t^b (s-t)^{\alpha-1}P_{\alpha}(s-t) z_s\, dB_s,
	\end{align}
	with the cost functional
	\begin{align}\label{e2}
		\mathcal{J}(u)=\mathbb{E}h(x_0)+\frac{1}{\Gamma(\alpha)}\, \mathbb{E}\int_0^b (b-s)^{\alpha-1}l(s,x_s,z_s,u_s)\, ds.
	\end{align}
	Here, 
	\[
	\kappa :[0,b]\times H \times L^2_0\times U \to H,\quad 
	l:[0,b]\times H \times L^2_0\times U \to \mathbb{R},\quad
	h: H\to H
	\]
	are measurable functions, $\xi \in L^2(\Omega,\mathcal{F}_b,H)$, $u:[0,b]\times \Omega\to U$, and $\alpha\in (\frac{1}{2},1)$.
	
	The following assumptions are imposed.
	\begin{assumption}\label{as1}
		\leavevmode
		\begin{enumerate}
			\item[(i)] $\kappa,l,h$ are continuously differentiable in $(x,z,u)$.\\
			\item[(ii)] The mappings $\kappa_x$
			and $\kappa_z$ are uniformly bounded: 
			\[\Vert \kappa_x\Vert+\Vert \kappa_z\Vert \leq C,\]
			and 
			\[
			\Vert l_x\Vert+\Vert h_x\Vert \leq C(1+\Vert x\Vert ),\quad  
			\Vert l_z\Vert \leq C(1+\Vert z\Vert ),
			\]
			where $C>0$.
		\end{enumerate}
	\end{assumption}
	
	Define the admissible control set
	\[
	\mathcal{U}_{ad}=\{u\in L^2_{\mathcal{F}}([0,b],U): \, u(t,\omega)\in U\}.
	\]
	It is clear that under assumptions \ref{as1}, for any $u \in \mathcal{U}_{ad}$ the state equation \eqref{e1} admits a unique solution 
	\[
	(x,z) = \bigl(x(\cdot,u_{(\cdot)}),\, z(\cdot,u_{(\cdot)})\bigr),
	\]
	and the cost functional \eqref{e2} is well defined. We call $(x,z,u)$ an admissible triple, and $(x,z)$ an \emph{admissible state process}. 
	
	\paragraph{Problem A.} Find a control $u^0_{(\cdot)}\in \mathcal{U}_{ad}$ such that
	\begin{equation}\label{29}
		\mathcal{J}(u^0) = \inf_{u\in \mathcal{U}_{ad}} \mathcal{J}(u).
	\end{equation}
	Any control $u^0$ satisfying \eqref{29} is called an optimal control. 
	The corresponding processes 
	\[
	(x^0,z^0) = \bigl(x(\cdot,u^0_{(\cdot)}),\, z(\cdot,u^0_{(\cdot)})\bigr),
	\]
	and the triple $(x^0,z^0,u^0)$ are called an \emph{optimal state process} and an \emph{optimal triple}, respectively.
	
	Assume that $(x^0_{(\cdot)},z^0_{(\cdot)},u^0_{(\cdot)})$ is an optimal solution of the control problem \eqref{e1}–\eqref{e2}. 
	Consider the following forward stochastic equation:
	\begin{align}\label{30}
		\Psi_t=S_{\alpha}(t)h_{x}(x^0_0)&+\int_0^t (t-s)^{\alpha-1}P_{\alpha}(t-s)\left( \kappa^*_x [s]\,\psi_s+l_x[s]\right)\, ds\nonumber\\
		&+\int_0^t (t-s)^{\alpha-1}P_{\alpha}(t-s) \left( \kappa^*_z [s]\,\psi_s+l_z[s]\right)\, dB_s.
	\end{align}
	
	\noindent
	We will use the following notation:
	\begin{align}
		\begin{cases}
			\mathscr{K} [t] = \mathscr{K}\bigl(t,x^0_t,z^0_t,u^0_t\bigr),\\[1mm]
			\Delta_u \mathscr{K}_t = \mathscr{K}\bigl(t,x^0_t,z^0_t,u_t\bigr)-\mathscr{K}[t],\\[1mm]
			\Delta_x \mathscr{K}_t = \mathscr{K}\bigl(t,x_t,z^0_t,u^0_t\bigr)-\mathscr{K}[t],\\[1mm]
			\Delta_z \mathscr{K}_t = \mathscr{K}\bigl(t,x^0_t,z_t,u^0_t\bigr)-\mathscr{K}[t].
		\end{cases}
	\end{align}
	
	Let $\mathcal{H}$ be the Hamiltonian function
	\[
	\mathcal{H}(t,v) = \mathbb{E}\Big\langle \kappa\bigl(t,x^0_t,z^0_t,v\bigr),\, \psi_t\Big\rangle 
	- \frac{(b-t)^{\alpha-1}}{\Gamma(\alpha)}l\bigl(t,x^0_t,z^0_t,v\bigr).
	\]
	
	For any $v\in \mathcal{U}_{ad}$, $t_0\in[0,b)$ and $0<\varepsilon\leq b-t_0$, 
	define a \emph{spike variational control} by
	\[
	u^{\varepsilon}_t = 
	\begin{cases}
		v, & t \in [t_0,\,t_0+\varepsilon], \\[1mm]
		u^0_t, & \text{otherwise}.
	\end{cases}
	\]
	
	Let $(x^{\varepsilon}_{(\cdot)},z^{\varepsilon}_{(\cdot)})$ be the solution of \eqref{e1} corresponding to $u^{\varepsilon}_{(\cdot)}$, and let $(p^{\varepsilon}_{(\cdot)},q^{\varepsilon}_{(\cdot)})$ be the solution of the linear BSDE
	\begin{align}\label{ad1}
		p^{\varepsilon}_t&=\int_t^b (s-t)^{\alpha-1}P_{\alpha}(s-t)\kappa_x[s]p^{\varepsilon}_s\, ds+\int_t^b (s-t)^{\alpha-1}P_{\alpha}(s-t) \kappa_z[s]q^{\varepsilon}_s\, ds\nonumber\\
		&+\int_t^b (s-t)^{\alpha-1}P_{\alpha}(s-t)\Delta u_{\varepsilon}\kappa[s]\, ds+\int_t^b (s-t)^{\alpha-1}P_{\alpha}(s-t)q^{\varepsilon}_s\, dB_s.
	\end{align}
	
	\begin{theorem}
		Let Assumption \ref{as1} hold. Then
		\begin{align}\label{mom}\begin{cases}
				\sup_{0\le t\le b}\mathbb{E}\|p^{\varepsilon}_t\|^2 + \mathbb{E}\int_0^b \|q^{\varepsilon}_t\|^2 dt = O(\varepsilon^{2\alpha}),\\
				\sup_{0\le t\le b}\mathbb{E}\|p^{\varepsilon}_t\|^4 + \mathbb{E}\int_0^b \|q^{\varepsilon}_t\|^4 dt = O(\varepsilon^{4\alpha}),\\
				\sup_{0\le t\le b}\mathbb{E}\|x^{\varepsilon}_t-x^0_t-p^{\varepsilon}_t\|^2 + \mathbb{E}\int_0^b \|z^{\varepsilon}_t-z^0_t-q^{\varepsilon}_t\|^2 dt = o(\varepsilon^{2\alpha}).
			\end{cases}
		\end{align}
		Moreover, the first variation formula holds:
		\begin{align}\label{variation-fractional}
			&\qquad\mathcal{J}(u^{\varepsilon})-\mathcal{J}(u^0)
			= \mathbb{E}\langle h_x(x^0_0),\, p^{\varepsilon}_0\rangle
			+\frac{1}{\Gamma(\alpha)}\mathbb{E}\int_0^b (b-s)^{\alpha-1}\langle l_x[s],\, p^{\varepsilon}_s\rangle\, ds\nonumber\\
			&+\frac{1}{\Gamma(\alpha)}\mathbb{E}\int_0^b (b-s)^{\alpha-1}\langle l_z[s],\, q^{\varepsilon}_s\rangle\, ds
			+\frac{1}{\Gamma(\alpha)}\mathbb{E}\int_0^b (b-s)^{\alpha-1}\Delta l_u[s]\, ds + o(\varepsilon^{\alpha}).
		\end{align}
	\end{theorem}
	
	\begin{proof}
		Using \eqref{ad1} and Jensen’s inequality, we get
		\begin{align*}
			\mathbb{E}\|p^{\varepsilon}_t\|^2 &\le 4\,\mathbb{E}\Big\|\int_t^b (s-t)^{\alpha-1}P_{\alpha}(s-t) \kappa_x[s]\, p^{\varepsilon}_s\,ds\Big\|^2
			+ 4\,\mathbb{E}\Big\|\int_t^b (s-t)^{\alpha-1}P_{\alpha}(s-t) \kappa_z[s]\, q^{\varepsilon}_s\,ds\Big\|^2\\
			& + 4\,\mathbb{E}\Big\|\int_t^b (s-t)^{\alpha-1}P_{\alpha}(s-t)\Delta_u \kappa[s]\,ds\Big\|^2
			+ 4\,\mathbb{E}\Big\|\int_t^b (s-t)^{\alpha-1}P_{\alpha}(s-t) q^{\varepsilon}_s\,dB_s\Big\|^2.
		\end{align*}
		Using lemma \ref{lem2}, assumption \ref{as1} and Cauchy--Schwarz inequality we obtain
		\begin{align*}
			\mathbb{E}\Big\|\int_t^b (s-t)^{\alpha-1}P_{\alpha}(s-t) \kappa_x[s]\, p^{\varepsilon}_s\,ds\Big\|^2 
			&\leq \mathbb{E}\Bigg( \int_t^b (s-t)^{\alpha-1}\Vert P_{\alpha}(s-t)\Vert \Vert \kappa_x[s]\Vert \Vert p^{\varepsilon}_s\Vert\, ds \Bigg)^2\\
			&\leq \Bigg(\frac{\alpha M_1C}{\Gamma(1+\alpha)}\Bigg)^2 \mathbb{E}\Bigg( \int_t^b (s-t)^{\alpha-1} \Vert p^{\varepsilon}_s\Vert\, ds \Bigg)^2\\
			&\leq  \Bigg(\frac{\alpha M_1C}{\Gamma(1+\alpha)}\Bigg)^2 \frac{(b-t)^{2\alpha-1}}{2\alpha-1}\int_t^b \mathbb{E}\Vert p^{\varepsilon}_s\Vert^2\, ds\\
			&\leq C_1 \int_t^b \mathbb{E}\Vert p^{\varepsilon}_s\Vert^2\, ds,
		\end{align*}
		where 
		\begin{align*}
			C_1=  \Bigg(\frac{\alpha M_1C}{\Gamma(1+\alpha)}\Bigg)^2 \frac{b^{2\alpha-1}}{2\alpha-1}.
		\end{align*}
		And similarly, we obtain
		\[
		\mathbb{E}\Big\|\int_t^b (s-t)^{\alpha-1}P_{\alpha}(s-t) \kappa_z[s]\, q^{\varepsilon}_s\,ds\Big\|^2 \leq C_1\int_t^b \mathbb{E}\|q^{\varepsilon}_s\|^2 ds.
		\]
		Since \(\Delta_u \kappa_s=0\) for \(s\notin[t_0,t_0+\varepsilon]\) and \(\Delta_u \kappa\) is bounded on the spike,
		\begin{align}\label{est:spike}
			\Big\|\int_t^b (s-t)^{\alpha-1}P_\alpha(s-t)\Delta_u \kappa_s\,ds\Big\|
			&\leq \frac{CM_1\alpha}{\Gamma(1+\alpha)} \int_{t_0}^{t_0+\varepsilon} (s-t)^{\alpha-1} ds =\frac{CM_1}{\Gamma(1+\alpha)}\varepsilon^{\alpha}.
		\end{align}
		
		Thus
		\begin{equation}\label{est:spike2}
			\mathbb{E}\Big\|\int_t^b (s-t)^{\alpha-1}P_\alpha(s-t)\Delta_u \kappa_s\,ds\Big\|^2
			\leq C_2\varepsilon^{2\alpha},
		\end{equation}
		where 
		\begin{align*}
			C_2=\left(\frac{CM_1}{\Gamma(1+\alpha)}\right)^2.
		\end{align*}
		
		Use Itô isometry for the stochastic part, we obtain
		\begin{align}\label{est:ito}
			&\mathbb{E}\Big\|\int_t^b (s-t)^{\alpha-1}P_\alpha(s-t)q^\varepsilon_s\,dB_s\Big\|^2 ds
			\le C_3\int_t^b (s-t)^{2\alpha-2}\mathbb{E}\|q^\varepsilon_s\|^2 ds,
		\end{align}
		where 
		\begin{align*}
			C_3=\left(\frac{\alpha M_1}{\Gamma(1+\alpha)}  \right)^2 .
		\end{align*}
		
		Then, there exists \(C_4\) such that
		\begin{equation}\label{volterra1}
			\mathbb{E}\|p^\varepsilon_t\|^2
			\leq C_4\int_t^b\mathbb{E}\|p^\varepsilon_s\|^2\, ds+C_4\int_t^b\mathbb{E}\|q^\varepsilon_s\|^2\, ds+ C_4\int_t^b (s-t)^{2\alpha-2}\mathbb{E}\|q^\varepsilon_s\|^2\, ds + C_4\varepsilon^{2\alpha},
		\end{equation}
		where 
		\begin{align*}
			C_4=\max_{i=1,2,3}\left\{C_i\right\}.
		\end{align*}
		
		Integrate \eqref{volterra1} with respect to \(t\) over \([0,b]\) and change order using Fubini:
		\begin{align*}
			\int_0^b \mathbb{E}\|p^\varepsilon_t\|^2 dt
			&\le C_4\int_0^b \int_t^b \mathbb{E}\|p^\varepsilon_s\|^2\, ds\, dt+ C_4\int_0^b \int_t^b \mathbb{E}\|q^\varepsilon_s\|^2\, ds\, dt\\
			&+C_4\int_0^b \int_t^b  (s-t)^{2\alpha-2}\mathbb{E}\|q^\varepsilon_s\|^2 \,ds\, dt + C_4 b\varepsilon^{2\alpha}\\
			&= C_4 \int_0^b s\mathbb{E}\Vert p^{\varepsilon}_s\Vert^2\, ds+ C_4 \int_0^b s\mathbb{E}\Vert q^{\varepsilon}_s\Vert^2\, ds\\
			&+C_4 \int_0^b \frac{s^{2\alpha-1}}{2\alpha-1}\mathbb{E}\Vert q^{\varepsilon}_s\Vert^2\, ds+ C_4 b\varepsilon^{2\alpha}\\
			&\leq C_4b \int_0^b \mathbb{E}\Vert p^{\varepsilon}_s\Vert^2\, ds+ C_4\left(b+\frac{b^{2\alpha-1}}{2\alpha-1}\right)\int_0^b \mathbb{E}\Vert q^{\varepsilon}_s\Vert^2\, ds+ C_4 b\varepsilon^{2\alpha},
		\end{align*}
		where 
		\begin{align*}
			\widetilde{C}=\max\left\{ C_4b,\, C_4\left(b+\frac{b^{2\alpha-1}}{2\alpha-1}\right) \right\}.
		\end{align*}
		Define
		\[
		\Phi_t:=\sup_{t\le r\le b}\mathbb{E}\|p^\varepsilon_r\|^2 + \mathbb{E}\int_t^b \mathbb{E}\|q^\varepsilon_s\|^2 ds.
		\]
		From \eqref{volterra1} and the above integral bound we obtain 
		\[
		\Phi_t\leq  \widetilde{C}\int_t^b \Phi_s\,ds +  \widetilde{C}\varepsilon^{2\alpha},\qquad 0\le t\le b.
		\]
		Applying Gronwall's inequality on \([0,b]\) yields
		\[
		\Phi_0\le  \widetilde{C}\varepsilon^{2\alpha},
		\]
		i.e.
		\[
		\sup_{0\le t\le b}\mathbb{E}\|p^\varepsilon_t\|^2 + \mathbb{E}\int_0^b \|q^\varepsilon_s\|^2 ds = O(\varepsilon^{2\alpha}).
		\]

		Next, we show the fourth-moment estimate. Using \eqref{ad1}, the Burkholder--Davis--Gundy inequality, Hölder inequality and Lemma \ref{lem2}, we have
		\begin{align*}
			\mathbb{E}\|p^\varepsilon_t\|^4 &\le 64\,\mathbb{E}\Big\|\int_t^b (s-t)^{\alpha-1} P_\alpha(s-t) \kappa_x[s] p^\varepsilon_s\, ds \Big\|^4
			+ 64\,\mathbb{E}\Big\|\int_t^b (s-t)^{\alpha-1} P_\alpha(s-t) \kappa_z[s] q^\varepsilon_s\, ds \Big\|^4\\
			& + 64\,\mathbb{E}\Big\|\int_t^b (s-t)^{\alpha-1} P_\alpha(s-t) \Delta_u \kappa[s]\, ds \Big\|^4
			+ 64\,\mathbb{E}\Big\|\int_t^b (s-t)^{\alpha-1} P_\alpha(s-t) q^\varepsilon_s\, dB_s \Big\|^4.
		\end{align*}
		
		For the deterministic integrals, by Lemma \ref{lem2}, Assumption \ref{as1} and Hölder inequality, we obtain
		\begin{align*}
			\mathbb{E}\Big\|\int_t^b (s-t)^{\alpha-1} P_\alpha(s-t) \kappa_x[s] p^\varepsilon_s\, ds \Big\|^4&\leq \left(\frac{CM_1\alpha}{\Gamma(1+\alpha)}\right)^4\mathbb{E}\left\{  \int_t^b (s-t)^{\alpha-1}\Vert p^{\varepsilon}_s\Vert\, ds \right\}^4\\
			&\leq\left(\frac{CM_1\alpha}{\Gamma(1+\alpha)}\right)^4\left(\int_t^b (s-t)^{\frac{4}{3}(\alpha-1)}\, ds\right)^3\int_t^b \mathbb{E}\Vert p^{\varepsilon}_s\Vert^4\, ds\\
			&\leq C_1 \int_t^b \mathbb{E}\|p^\varepsilon_s\|^4 ds,
		\end{align*}
		and similarly
		\[
		\mathbb{E}\Big\|\int_t^b (s-t)^{\alpha-1} P_\alpha(s-t) \kappa_z[s] q^\varepsilon_s\, ds \Big\|^4
		\le C_1 \int_t^b \mathbb{E}\|q^\varepsilon_s\|^4 ds,
		\]
		where 
		\begin{align*}
			C_1=\left(\frac{CM_1\alpha}{\Gamma(1+\alpha)}\right)^4 \frac{27\,b^{4\alpha-1}}{(4\alpha-1)^3}.
		\end{align*}
		
		For the spike integral, using \eqref{est:spike} and the bound of \(P_\alpha\),
		\[
		\mathbb{E}\Big\|\int_t^b (s-t)^{\alpha-1} P_\alpha(s-t) \Delta_u \kappa[s]\, ds\Big\|^4 \le C_2 \varepsilon^{4\alpha},
		\]
		where 
		\begin{align*}
			C_2 := \Big(\frac{ M_1 C}{ \Gamma(1+\alpha)}\Big)^4.
		\end{align*}
		For the stochastic integral, by the Burkholder--Davis--Gundy, Cauchy-Shwartz inequalities and Lemma \ref{lem2},
		\begin{align*}
			\mathbb{E}\Big\|\int_t^b (s-t)^{\alpha-1} P_\alpha(s-t) q^\varepsilon_s\, dB_s \Big\|^4
			&\le\Big(\frac{\alpha M_1}{\Gamma(1+\alpha)}\Big)^4 \, \mathbb{E} \Big(\int_t^b (s-t)^{2\alpha-2} \|q^\varepsilon_s\|^2 ds \Big)^2 \\
			&\le\Big(\frac{\alpha M_1}{\Gamma(1+\alpha)}\Big)^4 \, \int_t^b (s-t)^{2\alpha-2}\, ds\,\int_t^b   (s-t)^{2\alpha-2} \mathbb{E} \|q^\varepsilon_s\|^4 ds  \\
			&\le C_3 \,\int_t^b  (s-t)^{2\alpha-2}\mathbb{E}\|q^\varepsilon_s\|^4 ds,
		\end{align*}
		with
		\[
		C_3 :=  \frac{b^{2\alpha-1}}{2\alpha-1} \Big(\frac{\alpha M_1}{\Gamma(1+\alpha)}\Big)^4.
		\]

		Integrating the above estimates with respect to \(t\in[0,b]\) and swapping the order using Fubini’s theorem, we get
		
		\begin{align*}
			\int_0^b \mathbb{E}\|p^\varepsilon_t\|^4 dt
			&\le C_1\int_0^b  \int_t^b \mathbb{E}\|p^\varepsilon_s\|^4 ds\, dt 
			+  C_1\int_0^b \int_t^b \mathbb{E}\|q^\varepsilon_s\|^4 ds\, dt   \\
			&+C_3\int_0^b \int_t^b (s-t)^{2\alpha-2}\mathbb{E}\|q^\varepsilon_s\|^4 ds\, dt   +C_2b \varepsilon^{4\alpha} \\
			&= C_1 \int_0^b s\, \mathbb{E}\|p^\varepsilon_s\|^4 ds 
			+ C_1\, \int_0^bs\, \mathbb{E}\|q^\varepsilon_s\|^4 ds \\
			&+C_3\, \int_0^b \frac{s^{2\alpha-1}}{2\alpha-1}\mathbb{E}\Vert q^{\varepsilon}_s\Vert^4\, ds+ C_2 b \varepsilon^{4\alpha} \\
			&= C_1 \, b \int_0^b  \, \mathbb{E}\|p^\varepsilon_s\|^4 \,ds
			+ C_1\,b \int_0^b  \, \mathbb{E}\|q^\varepsilon_s\|^4 \,ds\\
			&+C_3\, \frac{b^{2\alpha-1}}{2\alpha-1}\int_0^b  \, \mathbb{E}\|q^\varepsilon_s\|^4 \,ds+ C_2 b \varepsilon^{4\alpha} \\
			&\le C_4 \int_0^b \mathbb{E}\|p^\varepsilon_s\|^4 ds
			+ C_4 \int_0^b \mathbb{E}\|q^\varepsilon_s\|^4 \,ds+ C_4 \varepsilon^{4\alpha},
		\end{align*}
		
		where we have defined
		\[
		C_4 := \max\left\{C_2b,\, C_1 b+C_3\, \frac{b^{2\alpha-1}}{2\alpha-1}\right\}.
		\]
		
		Define
		\[
		\tilde{\Phi}_t := \sup_{r\in[t,b]} \mathbb{E}\|p^\varepsilon_r\|^4 + \int_t^b \mathbb{E}\|q^\varepsilon_s\|^4 ds.
		\]
		
		Then, from the above inequality, we have 
		\[
		\tilde{\Phi}_t \le C_4 \int_t^b \tilde{\Phi}_s ds + C_4 \varepsilon^{4\alpha}, \qquad t\in[0,b].
		\]
		
		Applying Gronwall's inequality yields
		\[
		\sup_{0\le t \le b} \mathbb{E}\|p^\varepsilon_t\|^4 + \int_0^b \mathbb{E}\|q^\varepsilon_s\|^4 ds
		= O(\varepsilon^{4\alpha}).
		\]
		\medskip
		
		Next, define the remainder
		\[
		\tilde{x}^\varepsilon := x^\varepsilon - x^0 - p^\varepsilon, \qquad \tilde{z}^\varepsilon := z^\varepsilon - z^0 - q^\varepsilon.
		\]
		
		Then \((\tilde{x}^\varepsilon, \tilde{z}^\varepsilon)\) satisfies
		\begin{align*}
			\tilde{x}^\varepsilon_t &= \int_t^b (s-t)^{\alpha-1} P_\alpha(s-t) \Big(\kappa_x[s] \tilde{y}^\varepsilon_s + \kappa_z[s] \tilde{z}^\varepsilon_s + R^\varepsilon_s\Big) ds \\
			& + \int_t^b (s-t)^{\alpha-1} P_\alpha(s-t) \tilde{z}^\varepsilon_s dB_s,
		\end{align*}
		where
		\[
		R^\varepsilon_s := \Delta_x \kappa_s - \kappa_x[s] p^\varepsilon_s + \Delta_z \kappa_s - \kappa_z[s] q^\varepsilon_s.
		\]
		
		By Assumption \ref{as1} and the previous estimates, \(\mathbb{E}\|R^\varepsilon_s\|^2 = o(\varepsilon^{2\alpha})\). Using the same nested-integral technique as above, we obtain
		\[
		\sup_{0\le t \le b} \mathbb{E}\|\tilde{x}^\varepsilon_t\|^2 + \int_0^b \mathbb{E}\|\tilde{z}^\varepsilon_s\|^2 ds = o(\varepsilon^{2\alpha}).
		\]
		
		\medskip
		
		Finally, we conclude
		\begin{align*}
			&\sup_{0\le t \le b} \mathbb{E}\|x^\varepsilon_t - x^0_t - p^\varepsilon_t\|^2 + \int_0^b \mathbb{E}\|z^\varepsilon_s - z^0_s - q^\varepsilon_s\|^2 ds = o(\varepsilon^{2\alpha}).
		\end{align*}

		For variational formula, applying Taylor formula, we have
		\begin{align*}
			&\qquad\mathcal{J}(u^{\varepsilon})-\mathcal{J}(u^0)=\mathbb{E}\left[  h(x^{\varepsilon}_0-h(x^0_0)\right]\\
			&+\frac{1}{\Gamma(\alpha)}\mathbb{E}\int_0^b (b-s)^{\alpha-1}\left(l(s,x^{\varepsilon}_s,z^{\varepsilon}_s,u^{\varepsilon}_s)-l(s,x^{0}_s,z^{\varepsilon}_s,u^{\varepsilon}_s)\right)\, ds\\
			&+\frac{1}{\Gamma(\alpha)}\mathbb{E}\int_0^b (b-s)^{\alpha-1}\left(l(s,x^{0}_s,z^{\varepsilon}_s,u^{\varepsilon}_s)-l(s,x^{0}_s,z^{0}_s,u^{\varepsilon}_s)\right)\, ds\\
			&+\frac{1}{\Gamma(\alpha)}\mathbb{E}\int_0^b (b-s)^{\alpha-1}\left(l(s,x^{0}_s,z^{0}_s,u^{\varepsilon}_s)-l(s,x^{0}_s,z^{0}_s,u^{0}_s)\right)\, ds\\
			&=\mathbb{E}\int_0^1 \left\langle h_x\bigl(x^0_0+\delta (x^{\varepsilon}_0-x^0_0)\bigr), \, x^{\varepsilon}_0-x^0_0 \right\rangle d\delta \\
			&+\frac{1}{\Gamma(\alpha)}\mathbb{E}\int_0^b (b-s)^{\alpha-1}\Big\langle l_x\bigl(s,x^0_s+\delta (x^{\varepsilon}_s-x^0_s), z^{\varepsilon}_s,u^{\varepsilon}_s\bigr),\, x^{\varepsilon}_s-x^0_s\Big\rangle ds \\
			&+\frac{1}{\Gamma(\alpha)}\mathbb{E}\int_0^b (b-s)^{\alpha-1}\Big\langle l_z\bigl(s,x^0_s, z^0_s+\delta (z^{\varepsilon}_s-z^0_s), u^{\varepsilon}_s\bigr),\, z^{\varepsilon}_s-z^0_s\Big\rangle ds\\
			&+\frac{1}{\Gamma(\alpha)}\mathbb{E}\int_0^b (b-s)^{\alpha-1}\left(l(s,x^{0}_s,z^{0}_s,u^{\varepsilon}_s)-l(s,x^{0}_s,z^{0}_s,u^{0}_s)\right)\, ds.
		\end{align*}
		
		Define the remainders 
		\[
		X^{\varepsilon}_t = x^{\varepsilon}_t-x^0_t-p^{\varepsilon}_t,\qquad 
		Z^{\varepsilon}_t = z^{\varepsilon}_t-z^0_t-q^{\varepsilon}_t.
		\]
		Then we have
		\begin{align*}
			&\mathcal{J}(u^{\varepsilon})-\mathcal{J}(u^0)\\
			&=\mathbb{E}\langle h_x(x^{0}_0),p^{\varepsilon}_0\rangle +\mathbb{E}\langle h_x(x^{0}_0),X^{\varepsilon}_0\rangle\\
			&+\mathbb{E}\int_0^1 \left\langle h_x\left(x^0_0+\delta (x^{\varepsilon}_0 - x^0_0)\right)-h_x(x^0_0), p^{\varepsilon}_0+X^{\varepsilon}_0 \right\rangle\, d\delta\\
			&+\frac{1}{\Gamma(\alpha)}\, \mathbb{E} \int_0^b (b-s)^{\alpha-1}  \langle l_x[s], p^{\varepsilon}_s+X^{\varepsilon}_s\rangle\, ds\\
			&+\frac{1}{\Gamma(\alpha)}\, \mathbb{E} \int_0^b (b-s)^{\alpha-1}  \langle l_z[s], q^{\varepsilon}_s+Z^{\varepsilon}_s\rangle\, ds
			+\frac{1}{\Gamma(\alpha)}\, \mathbb{E}\int_0^b (b-s)^{\alpha-1}\Delta l_u[s]\, ds\\
			&+\frac{1}{\Gamma(\alpha)}\, \mathbb{E}\int_0^b (b-s)^{\alpha-1}  \left\langle l_x\bigl(s,x^0_s+\delta (x^{\varepsilon}_s-x^0_s), Z^{\varepsilon}_s,u^{\varepsilon}_s\bigr)-l_x[s], p^{\varepsilon}_s+X^{\varepsilon}_s \right\rangle\, ds\\
			&+\frac{1}{\Gamma(\alpha)}\, \mathbb{E}\int_0^b (b-s)^{\alpha-1}  \left\langle l_z\bigl(s,x^{\varepsilon}_s,z^0_s+\delta (z^{\varepsilon}_s-z^0_s),u^{\varepsilon}_s\bigr)-l_z[s], q^{\varepsilon}_s+Z^{\varepsilon}_s \right\rangle \,ds.
		\end{align*}
		\noindent
		Using the moment estimates \eqref{mom}
		and assumption \ref{as1}, we have
		\begin{align*}
			\begin{cases}
				\mathbb{E}\langle h_x(x^{0}_0),X^{\varepsilon}_0\rangle = o(\varepsilon^\alpha),\\
				\mathbb{E}\int_0^b (b-s)^{\alpha-1} \langle l_x[s],X^\varepsilon_s\rangle ds = o(\varepsilon^\alpha),\\
				\mathbb{E}\int_0^b (b-s)^{\alpha-1} \langle l_z[s],Z^\varepsilon_s\rangle ds = o(\varepsilon^\alpha),\\
				\mathbb{E}\int_0^1 \left\langle h_x\big(x^0_0+\delta (x^\varepsilon_0-x^0_0)\big)-h_x(x^0_0), p^\varepsilon_0+X^\varepsilon_0 \right\rangle d\delta = o(\varepsilon^\alpha),\\
				\mathbb{E}\int_0^b (b-s)^{\alpha-1} \left\langle l_x\big(s, x^0_s+\delta(x^\varepsilon_s-x^0_s), z^\varepsilon_s, u^\varepsilon_s\big)-l_x[s], p^\varepsilon_s+X^\varepsilon_s \right\rangle ds = o(\varepsilon^\alpha),\\
				\mathbb{E}\int_0^b (b-s)^{\alpha-1} \left\langle l_z\big(s, x^\varepsilon_s, z^0_s+\delta(z^\varepsilon_s-z^0_s), u^\varepsilon_s\big)-l_z[s], q^\varepsilon_s+Z^\varepsilon_s \right\rangle ds = o(\varepsilon^\alpha).
			\end{cases}
		\end{align*}
		
		Therefore, all remainder terms are of order $o(\varepsilon^\alpha)$, and we conclude the first variation formula
		\begin{align}\label{variation-fractional-final}
			\mathcal{J}(u^\varepsilon) - \mathcal{J}(u^0)
			&= \mathbb{E}\langle h_x(x^0_0), p^\varepsilon_0\rangle
			+ \frac{1}{\Gamma(\alpha)} \mathbb{E} \int_0^b (b-s)^{\alpha-1} \langle l_x[s], p^\varepsilon_s \rangle ds \nonumber\\
			& + \frac{1}{\Gamma(\alpha)} \mathbb{E} \int_0^b (b-s)^{\alpha-1} \langle l_z[s], q^\varepsilon_s \rangle ds
			+ \frac{1}{\Gamma(\alpha)} \mathbb{E} \int_0^b (b-s)^{\alpha-1} \Delta l_u[s] ds + o(\varepsilon^\alpha).
		\end{align}

	\end{proof}
	\begin{theorem}\label{w1}
		Assume that assumption \eqref{as1} hold, and let $(x^0, z^0, u^0)$ be an optimal triple of Problem A. Then, there exists a process $\psi$ satisfying \eqref{30} such that
		\begin{equation}\label{36}
			\mathcal{H}(t, v) \leq \mathcal{H}(t, u^0_t), \quad \forall v \in U, \ \text{a.e. } t \in [0,b], \ \mathbb{P}\text{-a.s.}
		\end{equation}
	\end{theorem}
	
	\begin{proof}
		By formula \eqref{variation-fractional}, we have
		\begin{align*}
			&\qquad J(u^\varepsilon) - J(u^0) 
			= \mathbb{E} \langle h_x(x^0_0), p^\varepsilon_0 \rangle 
			+ \frac{1}{\Gamma(\alpha)}\, \mathbb{E} \int_0^b (b-s)^{\alpha -1} \langle l_x[s], p^\varepsilon_s \rangle ds \\
			&+  \frac{1}{\Gamma(\alpha)}\, \mathbb{E} \int_0^b (b-s)^{\alpha -1}\langle l_z[s], q^\varepsilon_s \rangle ds 
			+ \frac{1}{\Gamma(\alpha)}\, \mathbb{E} \int_0^b (b-s)^{\alpha -1} \Delta l_u[s] ds + o(\varepsilon^{\alpha}).
		\end{align*}
		
		On the other hand,
		\begin{align*}
			&\qquad \mathbb{E} \langle h_x(x^0_0), p^\varepsilon_0 \rangle 
			+ \frac{1}{\Gamma(\alpha)}\, \mathbb{E} \int_0^b (b-s)^{\alpha -1} \langle l_x[s], p^\varepsilon_s \rangle ds \\
			&+ \frac{1}{\Gamma(\alpha)}\, \mathbb{E} \int_0^b (b-s)^{\alpha -1} \langle l_z[s], q^\varepsilon_s \rangle ds
			=  \frac{1}{\Gamma(\alpha)}\,\mathbb{E} \int_0^b (b-s)^{\alpha -1} \langle \Delta u^\varepsilon \kappa[s], p^\varepsilon_s \rangle ds.
		\end{align*}
		
		Thus, we obtain
		\begin{align*}
			0 \le J(u^\varepsilon) - J(u^0) 
			=\frac{1}{\Gamma(\alpha)}\, \mathbb{E} \int_0^b  (b-s)^{\alpha -1} \langle \Delta u^\varepsilon \kappa[s], p^\varepsilon_s \rangle ds + o(\varepsilon^{\alpha}),
		\end{align*}
		and from here we can easily deduce the variational inequality \eqref{36}.
	\end{proof}
	
	\section{Backward Linear Quadratic Problem}
	
	In this section, we study the LQ optimal control problem associated with the FBSEE. The goal is to minimize a quadratic cost functional subject to a fractional backward dynamic system.
	
	We begin by considering the following optimal control problem:
	\begin{equation}\label{flq2}
		J(u) = \mathbb{E}\,\langle G x_0,x_0\rangle
		+ \frac{1}{\Gamma(\alpha)}\,\mathbb{E}\int_0^b (b-t)^{\alpha-1} \langle \Pi u_t,u_t\rangle \, dt \to \min,
	\end{equation}
	subject to the fractional backward system
	\begin{equation}\label{flq1}
		\begin{cases}
			{^{C}_t}D^{\alpha}_b x_t = [A\, x_t + B\, u_t + C\, z_t] + z_t\, \dfrac{dB_t}{dt},\\[3pt]
			x(b)=\xi ,
		\end{cases}
	\end{equation}
	where $B:U\to H$, $C:L^2_0\to H$, and $\Pi:[0,b]\to L(U)$ are given operators.  
	We assume that $G=G^*$, and that $\Pi_t=\Pi^*_t\geq \gamma I$ for some $\gamma>0$.  
	
	Let $u^0$ denote the optimal control and $(x^0,z^0)$ be the corresponding optimal state processes.  
	To analyze optimality, we introduce a perturbation of the form
	\[
	u^\varepsilon=u^0+\varepsilon v,
	\]
	where $v$ is an admissible variation. The corresponding perturbed state is denoted by
	\[
	x^\varepsilon=x^0+\varepsilon\delta y+o(\varepsilon^{\alpha}), \qquad
	z^\varepsilon=z^0+\varepsilon\delta z+o(\varepsilon^{\alpha}),
	\]
	with the terminal condition $\delta x(b)=0$.  
	The first-order approximation leads to the following linearized (or variational) equation:
	\begin{align*}
		\begin{cases}
			{^{C}_t}D_b^\alpha \delta x_t=A\delta x_t+B v_t+C\delta z_t+\delta z_t\dfrac{dB_t}{dt},\\[3pt]
			\delta x_b=0.
		\end{cases}
	\end{align*}
	
	\subsection{First Variation of the Cost Functional}
	
	We now compute the first variation of the cost functional $J(u)$.  
	By definition, it is given by
	\[\begin{aligned}
		\delta J(v) &= \frac{d}{d\varepsilon}J(u^{\varepsilon})\Bigg|_{\varepsilon=0} \\
		&= \frac{d}{d\varepsilon}\left[\mathbb{E}\langle G x^\varepsilon_0, x^\varepsilon_0\rangle + \frac{1}{\Gamma(\alpha)}\mathbb{E}\int_0^b (b-t)^{\alpha-1} \langle \Pi_t u^\varepsilon_t, u^\varepsilon_t \rangle dt \right] \Bigg|_{\varepsilon=0} \\
		&= \frac{d}{d\varepsilon} \left[ \mathbb{E}\langle G (x^0_0 + \varepsilon \delta x_0), x^0_0 + \varepsilon \delta x_0\rangle \right] \Bigg|_{\varepsilon=0} \\
		& + \frac{d}{d\varepsilon} \left[ \frac{1}{\Gamma(\alpha)} \mathbb{E}\int_0^b (b-t)^{\alpha-1} \langle \Pi_t(u^0_t + \varepsilon v_t), u^0_t + \varepsilon v_t \rangle dt \right] \Bigg|_{\varepsilon=0} \\
		&= \frac{d}{d\varepsilon}\left[\mathbb{E}\langle G x^0_0, x^0_0 \rangle + \varepsilon \mathbb{E}\langle G x^0_0, \delta x_0 \rangle + \varepsilon \mathbb{E}\langle G \delta x_0, x^0_0 \rangle + \varepsilon^2 \mathbb{E}\langle G \delta x_0,\delta x_0 \rangle + o(\varepsilon^{\alpha}) \right] \Bigg|_{\varepsilon=0} \\
		& + \frac{d}{d\varepsilon} \left[ \frac{1}{\Gamma(\alpha)} \mathbb{E}\int_0^b (b-t)^{\alpha-1} \langle \Pi_t u^0_t, u^0_t \rangle dt \right. \\
		& + \left. \frac{ \varepsilon}{\Gamma(\alpha)} \mathbb{E}\int_0^b (b-t)^{\alpha-1} \left[\langle \Pi_t u^0_t, v_t \rangle + \langle \Pi_t v_t, u^0_t \rangle \right] dt \right. \\
		& + \left.  \frac{\varepsilon^2}{\Gamma(\alpha)} \mathbb{E}\int_0^b (b-t)^{\alpha-1}  \langle \Pi_t v_t, v_t \rangle \right] dt\Bigg|_{\varepsilon=0} 
	\end{aligned}\]
	\[\begin{aligned}
		&= \frac{d}{d\varepsilon}\left[\mathbb{E}\langle G x^0_0, x^0_0 \rangle + 2\varepsilon \mathbb{E}\langle G x^0_0, \delta x_0 \rangle + \varepsilon^2 \mathbb{E}\langle G \delta x_0, \delta x_0 \rangle + o(\varepsilon^{\alpha}) \right] \Bigg|_{\varepsilon=0} \\
		& + \frac{d}{d\varepsilon} \left[\frac{1}{\Gamma(\alpha)} \mathbb{E}\int_0^b (b-t)^{\alpha-1} \langle \Pi_t u^0_t, u^0_t \rangle dt \right. \\
		& + \left. \frac{ \varepsilon}{\Gamma(\alpha)} \mathbb{E}\int_0^b (b-t)^{\alpha-1} \left[\langle \Pi_t u^0_t, v_t \rangle + \langle \Pi_t v_t, u^0_t \rangle \right] dt\right.\\
		& + \left.  \frac{\varepsilon^2}{\Gamma(\alpha)} \mathbb{E}\int_0^b (b-t)^{\alpha-1}  \langle \Pi_t v_t, v_t \rangle \right] dt\Bigg|_{\varepsilon=0} \\
		&= 2 \mathbb{E}\langle G x^0_0, \delta x_0\rangle + \frac{2}{\Gamma(\alpha)} \mathbb{E}\int_0^b (b-t)^{\alpha-1} \langle \Pi_t u^0_t, v_t \rangle dt .
	\end{aligned}\]
	Hence, we arrive at the compact form
	\begin{align}\label{18}
		\delta J(v) = 2 \mathbb{E}\langle G x^0_0, \delta x_0\rangle 
		+ \frac{2}{\Gamma(\alpha)} \mathbb{E}\int_0^b (b-t)^{\alpha-1} 
		\langle \Pi_t u^0_t, v_t \rangle dt.
	\end{align}
	
	\subsection{Adjoint Equation}
	
	To express $\delta J(v)$ in terms of $v$ only, we now introduce the adjoint process $\psi_t$ through a fractional integration-by-parts argument \cite{Love}.  
	For sufficiently regular $\psi$ and $\delta y$, the following identity holds:
	\begin{equation}\label{IBP-start}
		\begin{aligned}
			\mathbb{E}\int_0^b \langle \psi_t,{^{C}_t}D_b^\alpha \delta x_t\rangle\,dt
			&=-\mathbb{E}\int_0^b \Big\langle \psi_t,{^{RL}_t}I_b^\alpha \delta x'_t\Big\rangle\,dt\\
			&=-\frac{1}{\Gamma(1-\alpha)}\mathbb{E}\int_0^b\int_t^b 
			\langle\psi_t,(s-t)^{-\alpha}\delta x'_s\rangle\,ds\,dt\\
			&=-\frac{1}{\Gamma(1-\alpha)}\mathbb{E}\int_0^b
			\left\langle \int_0^s (s-t)^{-\alpha}\psi_t\,dt,\ \delta x'_s\right\rangle ds\\
			&=-\mathbb{E}\int_0^b \langle {^{RL}_0}I_s^{1-\alpha}\psi_s,\delta x'_s\rangle\,ds\\
			& = -\mathbb{E}\big\langle {^{RL}_0}I_s^{1-\alpha}\psi_s,\delta x_s\big\rangle\Big|_{s=0}^{s=b}
			+\mathbb{E}\int_0^b \big\langle {^{RL}_0}D_s^{\alpha}\psi_s,\delta x_s\big\rangle\,ds.
		\end{aligned}
	\end{equation}
	
	Evaluating the boundary terms and using $\delta x(b)=0$, we retain only the contribution at $s=0$. Then we obtain, the following result
	\begin{equation}\label{IBP-final}
		\begin{aligned}
			\mathbb{E}\int_0^b \langle \psi_t,{^{C}_t}D_b^\alpha \delta x_t\rangle\,dt=
			\mathbb{E}\big\langle {^{RL}_0}I_s^{1-\alpha}\psi_s,\delta x_s\big\rangle\Big|_{s=0}
			+\mathbb{E}\int_0^b \big\langle {^{RL}_0}D_s^{\alpha}\psi_s,\delta x_s\big\rangle\,ds.
		\end{aligned}
	\end{equation}
	
	Next, we substitute the linearized dynamics into the left-hand side of \eqref{IBP-final}:
	\[
	\begin{aligned}
		\mathbb{E}\int_0^b \langle \psi_t,{^{C}_t}D_b^\alpha \delta x_t\rangle\,dt
		&=\mathbb{E}\int_0^b \langle \psi_t,A\delta x_t+B v_t+C\delta z_t\rangle\,dt
		+\mathbb{E}\int_0^b \langle \psi_t,\delta z_t\rangle\,dB_t.
	\end{aligned}
	\]
	Consequently,
	\[
	\mathbb{E}\int_0^b \langle \psi_t,{^{C}_t}D_b^\alpha \delta x_t\rangle\,dt
	=\mathbb{E}\int_0^b \langle A^*\psi_t+C^* \psi_t\, \frac{dB_t}{dt},\delta x_t\rangle\,dt
	+\mathbb{E}\int_0^b \langle B^*\psi_t,v_t\rangle\,dt.
	\]
	
	Comparing this result with \eqref{IBP-final} yields the relation
	\begin{equation}\label{equate-full}
		\begin{aligned}
			&\mathbb{E}\big\langle {^{RL}_0}I_t^{1-\alpha}\psi_t,\delta x_t\big\rangle\Big|_{t=0}
			+\mathbb{E}\int_0^b \big\langle {^{RL}_0}D_t^{\alpha}\psi_t,\delta x_t\big\rangle\,ds\\
			&=\mathbb{E}\int_0^b \langle A^*\psi_t+C^* \psi_t\, \frac{dB_t}{dt},\delta x_t\rangle\,dt
			+\mathbb{E}\int_0^b \langle B^*\psi_t,v_t\rangle\,dt.
		\end{aligned}
	\end{equation}
	
	\medskip
	\noindent
	Substituting this relation into \eqref{18} and collecting like terms, we obtain
	\[
	\begin{aligned}
		\delta J(v)
		&=\mathbb{E}\langle 2 G x^0_0-{^{RL}_0}I^{1-\alpha}_t\psi_t\big\vert_{t=0},\ \delta x_0\rangle\\
		& +\mathbb{E}\int_0^b \Big\langle -{^{RL}_0}D_t^{\alpha}\psi_t
		+A^*\psi_t+C^* \psi_t\, \frac{dB_t}{dt},\ \delta x_t\Big\rangle\,dt\\
		& +\mathbb{E}\int_0^b \Big\langle B^*\psi_t
		+\frac{2}{\Gamma(\alpha)}(b-t)^{\alpha-1}\Pi_tu^0_t,\ v_t\Big\rangle\,dt.
	\end{aligned}
	\]
	
	To ensure $\delta J(v)=0$ for all admissible $v$ (and arbitrary $\delta y$), we impose the adjoint conditions so that the coefficients of $\delta x_t$ and $\delta x_0$ vanish.  
	Hence, the adjoint variable $\psi_t$ satisfies
	\begin{align}\label{flq3}
		\begin{cases}
			{^{RL}_0}D_t^{\alpha}\psi_t
			=A^*\psi_t+C^* \psi_t\, \frac{dB_t}{dt},\\
			{^{RL}_0}I^{1-\alpha}_t\psi_t\Big\vert_{t=0}=2\,G\,x^0_0.
		\end{cases}
	\end{align}
	
	With these choices, the first variation simplifies to
	\[
	\delta J(v)=\mathbb{E}\int_0^b 
	\Big\langle B^*\psi_t+\frac{2}{\Gamma(\alpha)}(b-t)^{\alpha-1}\Pi_tu^0_t,\ v_t\Big\rangle\,dt.
	\]
	The stationarity condition $\delta J(v)=0$ for all $v$ implies
	\[
	B^*_t\psi_t+\frac{2}{\Gamma(\alpha)}(b-t)^{\alpha-1}\Pi_tu^0_t=0.
	\]
	Therefore, using the invertibility of $\Pi_t$, the optimal control is explicitly given by
	\[
	u^0_t=-\frac{\Gamma(\alpha)}{2}(b-t)^{1-\alpha}\Pi_t^{-1}B^*_t\psi_t.
	\]
	
	\begin{theorem}\label{thm:flq}
		There exists a unique optimal control $u^0$ for the problem \eqref{flq2}–\eqref{flq1}.  
		Moreover, the optimal control admits the following explicit representation:
		\begin{equation}\label{flq4}
			u^0_t=-\frac{\Gamma(\alpha)}{2}(b-t)^{1-\alpha}\Pi_t^{-1}B^*_t\psi_t.
		\end{equation}
	\end{theorem}
	
	\begin{proof}
		The cost functional \eqref{flq2} is clearly quadratic and positive definite under the assumptions on $G$ and $\Pi_t$.  
		Thus, an optimal control exists.  
		Furthermore, since the associated backward system satisfies \eqref{flq3}, by applying Theorem~\ref{w1}, we obtain
		\[
		\langle  B u^0_t, \psi_t \rangle - \frac{(b-t)^{\alpha-1}}{\Gamma(\alpha)}\langle \Pi_t u^0_t, u^0_t \rangle 
		\geq \langle B v, \psi_t \rangle - \frac{(b-t)^{\alpha-1}}{\Gamma(\alpha)}\langle \Pi_t v, v \rangle,
		\quad \forall v \in U.
		\]
		This inequality directly leads to the optimality condition \eqref{flq4}.  
		Hence, the control given by \eqref{flq4} uniquely satisfies the SMP and is therefore the optimal control.  
	\end{proof}
	
	\section*{Conclusion}
	
	This paper has established a comprehensive theoretical framework for optimal control of FBSEES in Hilbert spaces. The main contributions are summarized as follows:
	
	$\bullet$  We have derived a necessary optimality condition in the form of a stochastic maximum principle for a general class of non-linear fractional stochastic control problems. This was achieved by employing the spike variation method, rigorously estimating the moments of the resulting variational system, and introducing a novel adjoint process tailored to the fractional backward dynamics.
	
	$\bullet$ We obtained an explicit, closed-form representation of the optimal control for the backward LQ optimal control problem. The solution is characterized in terms of the adjoint process, which itself is governed by a forward stochastic differential equation involving Riemann-Liouville fractional derivatives.
	
	$\bullet$  Our work provides a unified framework that seamlessly integrates tools from fractional calculus, stochastic analysis, and infinite-dimensional control theory. The results extend classical control theory to systems with memory and long-range dependencies, modeled by fractional derivatives.
	
	This research opens several avenues for future work, including the study of fully coupled forward-backward fractional systems, problems with state constraints, and the numerical analysis of the derived optimality conditions for practical applications in fields like engineering and mathematical finance.
	
	\section*{Data Availability}
	No data were used to support this study.
	
	\section*{Conflicts of Interest}
	The author declares that there are no conflicts of interest.
	
	\section*{Funding} 
	No funding was received.

\end{document}